\documentclass[11pt]{article}
\usepackage{graphicx}
\usepackage{amssymb,amsmath,amsthm,subfig}
\usepackage[all]{xy}
\usepackage{color}
\usepackage{verbatim}

\newtheorem{theorem}{Theorem}[section]
\newtheorem{lemma}[theorem]{Lemma}
\newtheorem{proposition}[theorem]{Proposition}
\newtheorem{corollary}[theorem]{Corollary}

\newenvironment{remark}[1][Remark]{\begin{trivlist}
\item[\hskip \labelsep {\bfseries #1}]}{\end{trivlist}}

\renewcommand{\a}{\"a}

\newcommand{\R}{\mathbb{R}}
\newcommand{\Z}{\mathbb{Z}}
\newcommand{\coker}{\mbox{coker}}
\newcommand{\<}{\left\langle}
\renewcommand{\>}{\right\rangle}

\newcommand{\be}{\begin{equation}}
\newcommand{\ee}{\end{equation}}
\newcommand{\bea}{\begin{eqnarray}}
\newcommand{\eea}{\end{eqnarray}}

\newcommand{\ba}{\begin{array}}
\newcommand{\ea}{\end{array}}
\newcommand{\ran}{\mbox{ran}}

\begin{document}

\title{Asymptotic Stability of the Toda $m$-soliton}
\author{G.N. Benes, A. Hoffman and C.E. Wayne}
\maketitle
\abstract{We prove that multi-soliton solutions of the Toda lattice are both linearly and nonlinearly
stable.  Our proof uses neither the inverse spectral method nor the Lax pair of the model but instead
studies the linearization of the B\"acklund} transformation which links the ($m-1$)-soliton solution to the 
$m$-soliton solution.  We use this to construct a conjugation between the Toda flow linearized about
an $m$-solition solution and the Toda flow linearized about the zero solution, whose stability properties
can be determined by explicit calculation.

\section{Introduction}

Solitary waves or solitons are a common feature of infinite dimensional dispersive dynamical systems in which the
effects of dispersion are exactly off-set by the effects of the nonlinearity.   They are also important features
of physical systems ranging from water waves to nonlinear optics.  As such, questions of stability and
interaction of such waves are natural and important.  One class of systems which is amenable to 
study are lattices of coupled, nonlinear oscillators like the Fermi-Pasta-Ulam model.  In the present paper
we study the stability of multi-soliton ($m$-soliton) solutions of a particular lattice system,
the Toda system, which features an exponential interaction potential between adjacent oscillators.

More precisely, this paper establishes the stability of $m$-soliton solutions to the Toda lattice equations
\be \label{eq:Toda} 
\ddot{Q} =e^{-(Q-Q_-)}-e^{-(Q_+ - Q)},
\ee
where $Q=(Q_n)_{n\in\Z}$ is the vector of positions of unit masses in an infinite chain of particles with nearest-neighbor interaction potential $V(x)=e^{-x}-(1-x)$.  Here and throughout the paper $X_\pm$ denotes a left or right shift of $X$, i.e. $(X_\pm)_n=X_{n\pm 1}$.  A solution of $\eqref{eq:Toda}$ of the form $Q_n(t) = \phi(n-ct)$ is called a {\it traveling wave solution}.  In the case that $\phi$ is monotone and the limits $\phi(\pm \infty)$ are defined, we say that $\phi$ is a {\it kink} or {\it front}.  Since $\eqref{eq:Toda}$ is invariant with respect to adding a constant, we can, without loss of generality, restrict attention to the case $\phi(-\infty) = 0$.  A special class
of front solutions  of   $\eqref{eq:Toda}$, the solitons, are given explicitly by $Q^{1}_n(t;\kappa,\gamma) = \log\frac{\cosh(\kappa n - t\sinh \kappa + \gamma)}{\cosh(\kappa(n+1) - t \sinh \kappa + \gamma)} - \kappa$.  Note that they form a two-dimensional manifold parameterized by $\kappa$ and $\gamma$, with
asymptotic limit $\lim_{n \to \infty} Q^{1}_n(t;\kappa,\gamma) = -2 \kappa$.

It is well known that the Toda lattice is an integrable system and hence admits multi-soliton solutions \cite{toda:1975,flaschka:1974}.  To be concrete, in our context an {\it $m$-soliton solution} is a solution $Q^{m}$ for which there are $3m$ constants $\gamma_1^\pm,\cdots ,\gamma_m^\pm$ and $\kappa_1,\cdots,\kappa_m$ such that \[ \lim_{t \to \pm \infty} \left| Q^{m}_n(t) - \sum_{j=1}^m Q^{1}_n(t;\kappa_j,\gamma_j^\pm) \right| = 0.\]
In practice $(\gamma_1^+,\cdots , \gamma_m^+)$ are determined by $(\gamma_1^-,\cdots , \gamma_m^-)$ and $(\kappa_1,\cdots, \kappa_m)$ so that the space of $m$-soliton solutions is $2m$-dimensional.

The main result of this paper is that these $m$-soliton solutions are asymptotically stable in a sense that we will make precise in section 2.  We do not depend on the existence of a
Lax pair or inverse scattering method, but we do use in an essential way the existence of 
a B\a cklund transformation which relates $m$- and $(m-1)-$ solitons \cite{toda:1975} :
\be \label{eq:BT} \ba{l}
P + e^{-(Q' -Q - \kappa_m)} + e^{-(Q - Q'_- +\kappa_m)} = 2\cosh\kappa_m \\
P' + e^{-(Q' -Q - \kappa_m)} + e^{-(Q_+ -Q' + \kappa_m )} = 2\cosh\kappa_m. \\
\ea \ee
Here $(Q,P)$ is an $m$-soliton solution, $(Q',P')$ is an $(m-1)$ soliton solution, and the speed of the additional soliton in $(Q,P)$ is given by $\frac{\sinh\kappa_m}{\kappa_m}$ \cite{toda:1975b}. Given $(Q',P')$, equations (\ref{eq:BT}) together with the relations $\dot{Q} = P$ and $\dot{Q}' = P'$ determine $(Q,P)$ up to a constant of integration corresponding to the phase of the new soliton $\gamma_m$.  The corresponding $m$-soliton $Q$ has the asymptotic limits $Q_{-\infty} = 0$ and $Q_\infty = -2\sum_{j=1}^m \kappa_j$.

This paper contains two main results.  Theorem $\ref{thm:main}$ establishes linear stability for the $m$-soliton solution.  Theorem $\ref{thm:stab}$ establishes nonlinear stability.  The proof that linear stability implies nonlinear stability is based on the now classical theory of modulation equations as developed in \cite{pego:1994,friesecke:2002}.  

The proof of linear stability is based on a more recent idea of using the {\em linearization}
of the B\a cklund transformation to construct a conjugacy between the linearization of the 
flow of the Toda equations around an $m$-soliton and the linearization
around an $m-1$-soliton.   The existence of such a conjugacy between the linearization about 
a single soliton and the linearization about the zero solution was implicitly exploited in
 \cite{mann:1997} in the context of the KdV equation and developed in 
 detail for the Toda lattice in  \cite{mizumachi:2008}.  More recently the 
 general case of the linearization around an $m$-soliton of the KdV equation has been considered
 in \cite{mizumachi:2010}.

\subsection{Hamiltonian Formulation}
It is convenient to define $R=Q_+-Q=(S-I)Q$.   Here $S$ is the shift operator and
throughout the paper we will use the subscript ``$+$'' to denote the action
of this operator - i.e. $(Sx) = x_+$.
We also define $U=(R,P)$.  Then (\ref{eq:Toda}) can be rewritten in the Hamiltonian form
\be \label{eq:Hamiltonian} 
\dot{U}=JH'(U).
\ee
Here the symplectic operator $J$ is given by $J=\left(\begin{array}{cc} 0 & S-I \\ I-S^{-1} & 0 \\ \end{array}\right)$.  The Hamiltonian is defined by
$H=\sum_{n\in\Z} \left(\frac{1}{2}P_n^2 + V(R_n)\right)$ where the interaction potential is given by
$V(R)=e^{-R}-1+R$.  Define the weighted norm $\|x\|_a=\|x\|_{\ell^2_a}^2 = \sum_{n \in \Z} e^{2an}x_n^2$ and let $\ell^2_a$ denote the associated weighted Hilbert Space.  Note that we will use $\|x\|$ to denote the standard $\ell^2$ norm.

Let $\Phi(t,s)$ be the evolution operator associated with the linearized Toda flow $\dot{u}=JH''(U^m)u$,
where $U^m$ is the $m$-soliton solution written in terms of the $(R,P)$ variables.  
Define the following subspace of $\ell^2_a \times \ell^2_a$:
\[ \label{eq:XN}
X_m(s):=\{u\in\ell^2_a\times\ell^2_a | \< u,J^{-1}\partial_{\kappa_i}U^m(s) \> = \< u,J^{-1}\partial_{\gamma_i}U^m(s) \> = 0,\ i=1,\ldots,m \}. \]
Here $\langle \cdot, \cdot \rangle$ is the standard $\ell^2$ inner product.  The condition 
$\< x,J^{-1}y\> = 0$ can be thought of as a symplectic orthogonality condition, as discussed in \cite{friesecke:2002}. 

\subsection{Main Results}

We begin with our results on decay of solutions of the linearized Toda equations.
These results depend on the fact that dispersive waves in the Toda equation
(i.e. solutions of the equations linearized about the zero
solution) all propagate with speed less than or equal to one,
while all solitary waves move with speeds strictly greater than
one.  In particular, for the linearization of the equations
about the zero solution  an easy and explicit computation shows
that solutions decay exponentially in time in any of the weighted
Hilbert spaces $\ell_a^2$ if the weight translates to either
the left or right with a speed greater than one (a detailed proof
of this fact is provided in Lemma 3 of \cite{mizumachi:2008}).
Intuitively, one can imagine that the region in which the 
norm ``sees'' the solution just ``outruns'' the dispersive
perturbations.
If we now use the fact that the linearized B\"acklund transformation
conjugates the semigroup of the linearization about the zero
solution to the linearization about the multi-soliton
(provided the orthogonality conditions in \eqref{eq:XN}
are satisfied) we see that we also expect exponential decay
of solutions of the equations linearized about the multi-solitons,
provided the weight in the norm moves faster than the 
dispersive part of the solution (i.e., faster than one)
but slower than the solitons themselves.

\begin{theorem} \label{thm:main}
Let $\gamma_i \in \R$ and $\kappa_i > 0$ be given for $i = 1,\cdots , m$.   Define $\kappa_{min} =\min\{\kappa_i|i=1,\ldots,m\}$.  Let $c>1$, let $a\in(0,2\kappa_{min})$, and let $\beta:=ca-2\sinh(a/2)$.
Let $\Phi(t,x)$ be the evolution operator for the linearization of the Toda equations about an $m$-soliton
solution with parameters $\kappa_i$ and $\gamma_i$, $i= 1, \dots , m$.

Then there exists a constant $K>0$ such that for any $u_0\in X_m(s)$ and for all $t\geq s$,
\[ ||e^{a(n-ct-T)} \Phi(t,s) u_0|| \leq Ke^{-\beta(t-s)}||e^{a(n-cs-T)}u_0||. \]
\end{theorem}

\begin{remark}
We note the perturbation only decays for $\beta > 0$, which corresponds to choices of $a$ and $c$ such that $c>\frac{\sinh{a/2}}{a/2}$.
\end{remark}

\begin{theorem} \label{thm:stab} Let $c > 1$, $\kappa$ and $\beta$ be as in Theorem $\ref{thm:main}$.  For each $a \in (0,2\kappa)$ and $\beta' \in (0,\beta)$, there is a $\delta > 0$  with
the property that if $u_0$ satisfies
\[ \| e^{a(\cdot - T)}u_0 \| + \|u_0\| < \delta,\]
for some $T \in \R$, then associated to the  solution $u(t)$ of 
Toda equations with initial data $u_0+U^m(t_0;\kappa_1,\dots,\gamma_m)$,
there exist real valued functions of a real variable $\gamma_i(t)$ and $\kappa_i(t)$ for $i = 1,\cdots , m$ such that 
\[\| e^{a(\cdot - ct-T)}[u(t) - U^m(t;\kappa(t),\gamma(t))]\| \le Ke^{-\beta' (t-t_0)}\]
and 
\[ \| u(t) - U^m(t;\kappa(t),\gamma(t))\|^2 + \sum_{i=1}^m |\kappa_i(t) - \kappa_i(t_0)| + |\gamma_i(t) - \gamma_i(t_0)|\le K\|u(t_0)-U^m(t_0;\kappa(t_0),\gamma(t_0))\|^2.\]
Moreover, the limits $\lim_{t \to \infty} \gamma_i(t) = \gamma_i(\infty) $ and $\lim_{t \to \infty} \kappa_i(t) = \kappa_i(\infty)$ exist.
Furthermore, the same result holds in backward time upon replacing $a$ with $-a$.
\end{theorem}

\begin{remark}
Note that we restrict our initial perturbation to lie in a weighted space.  At first blush, this might appear to limit our analysis to initial perturbations which are exponentially localized behind the $m$-soliton solution.  Note, however, that by varying the decay $a$ and the translation $T$ we can cover almost any initial data which is small in $\ell^2$ and decays at some exponential rate at spatial $\infty$.  For example, by choosing $T$ large, we can allow for an initial perturbation which is localized in front of the $m$-soliton solution.  By choosing $a$ small, we can allow this perturbation to have long tails.  

Note that we do disallow perturbations which decay very slowly, such as an infinite series of solitons, as were used in \cite{martel:2005} to construct a solution (to the gKdV equation) whose phase drifted off to $\infty$.  
\end{remark}

Stability results for multi-soliton solutions in integrable systems date back to Maddocks and Sachs \cite{maddocks:1993} where KdV multi-solitons are shown to be constrained minimizers of a Lyapunov function.  
{\color{black} While
a similar approach was shown to work for NLS-type equations \cite{kapitula:2007},
we know of no comparable result for the Toda lattice. }
 An alternate approach, developed (for example) in \cite{martel:2002} decomposes the multi-soliton solution into several weakly interacting pieces, each of which is localized about a soliton and then leverages stability theory for a single solitary wave.  This approach does extend to the Toda lattice, but its use is restricted to initial data corresponding to a train of solitons which will not interact in forward time.  On the contrary, we are able to show stability for $m$-soliton solutions, even when the interaction between two or more constituent solitons lies in the future.  

One drawback of the B\"{a}cklund transform method that we use here, is that unlike the variational/dispersive estimate approach in \cite{martel:2002} it does not extend easily to the non-integrable case.  
{\color{black} However, there
has been some recent progress in extending this type of idea
to  near-integrable systems, and in particular to the Fermi-Pasta-Ulam (FPU) lattice in the long-wave, low-amplitude regime.  The second two authors have studied this problem for the case of a single soliton \cite{hoffman:preprint} and for the case of two counter-propagating solitary waves \cite{hoffman:2008}
and very recently in  \cite{mizumachi:2010}, Mizumachi used linearized B\"acklund transformations
for  the KdV equation
linearized about multi-soliton
solutions to help construct asymptotic $m$-solitons in the FPU model.
}

\section{Linear stability of the Toda $m$-soliton solution}
In \cite{mizumachi:2008} the authors used the B\a cklund transformation along with its linearizations about the zero- and one-soliton solutions to show that linear stability of the zero solution of the
Toda lattice implies the linear stability of the one-soliton solution.  In this section we show
that this can be regarded as the first step in an induction argument: the stability of the $(m-1)$-soliton solution implies that of the $m$-soliton solution.  To proceed we must show (i) that the linearized B\a cklund transformation commutes with the linearized Toda flow, (ii) that the linearized B\a cklund transformation preserves orthogonality with the neutral modes of the linearized Toda system, and (iii) that the linearized B\a cklund transformation is an isomorphism between two spaces (defined below) that depend on these neutral modes. These facts tell us that for each $m$, the diagram in Figure \ref{fig:commuting} commutes.

\begin{figure}[ht]
\begin{center}
\xymatrix{
& & & {u_m(s)\in X_m(s)} \ar@{->}[rrr]_{\Phi_m(t,s)} & & & {u_m(t)\in X_m(t)}\\
 & & & & & & \\
& & & {u_{m-1}(s)\in X_{m-1}(s)} \ar@{->}[rrr]_{\Phi_{m-1}(t,s)} \ar@{->}[uu]^{B_m(s)}& & & {u_{m-1}(t)\in X_{m-1}(t)} \ar@{->}[uu]_{B_m(t)}\\
}
\caption{Commuting diagram used in the induction step}
\label{fig:commuting}
\end{center}
\end{figure}

In Figure \ref{fig:commuting}, $B_m(\tau)$ is an isomorphism derived from the linearized B\a cklund transform at time $\tau$, $\Phi_k(t,s)$ is the evolution according to the linearization of the Toda equation about the solution $U^k$, and
\be \label{eq:subspace}
X_k(\tau):=\{u \in\ell^2_a\times\ell^2_a | \<u ,J^{-1}\partial_{\gamma_i} U^k(\tau)\> = \<u ,J^{-1}\partial_{\kappa_i} U^k(\tau)\> = 0,\ i=1,\ldots,k \}
\ee
where $U^k$ is the $k$-soliton solution in our hierarchy of solutions related by the B\a cklund transform.  We must require these orthogonality conditions on $X_k$ because {\color{black} any perturbation which merely changed
the relative amplitude or phase of the different component solitons results in a new multi-soliton solution which
cannot converge back to the original one.  All other perturbations save those of this type, however,
should decay.}

Throughout the remainder of this section, $(Q',P')$ and $(Q,P)$ correspond to $U_{m-1}$ and $U_m$, respectively, for some $ m$.  Finally, let $(q',p')$ and $(q,p)$ denote solutions to the linearizations of the Toda differential equations about $(Q',P')$ and $(Q,P)$, respectively.

There will be several places where we make an argument that uses the fact that an $m$-soliton solution converges to the sum of $m$ constituent solitons as $t$ tends to either plus or minus infinity.
  Although this fact is fundamental and well-known to experts, we are not aware of a proof that exists in the literature, 
  so we state and prove the following:
  
\begin{lemma}[Resolution of $m$-soliton solutions into one-solitons] \label{lem:resolution}
An $m$ soliton solution converges to a sum of solitons exponentially fast in $\ell^1$ as $t \to \pm \infty$.  More concretely:
Let $Q^m(t;\kappa_1,\cdots,\kappa_m,\gamma_1,\cdots,\gamma_m)$ denote the $m$-soliton solution.  Then there exist phase shifts $\zeta_1^\pm,\cdots , \zeta_m^\pm$ such that for any $\delta \in (0,\sinh \kappa_1 t)$ we have 
\be \lim_{t \to \pm \infty} e^{\delta t} \left\| Q^m(t;\kappa_1,\cdots, \gamma_m) - \sum_{j=1}^m Q^1(t;\kappa_i,\gamma_i + \zeta_i^\pm) \right\|_{\ell^1} = 0. \label{eq:res} \ee
\end{lemma}

\begin{proof}  See Appendix. \end{proof}

\subsection{Linearized B\a cklund transformation commutes with linearized Toda flow}

It is convenient in this section to define new variables
\be \alpha:=e^{-(Q'-Q-\kappa_m)},\ \ \beta:=e^{-(Q-Q'_-+\kappa_m)}. \label{eq:ab} \ee
With these variables we can rewrite the Toda equations as
\be \label{eq:Toda1} \ba{ll}
\dot{Q} = P & \dot{Q}'=P'\\
\dot{P} = \alpha_-\beta - \alpha\beta_+ & \dot{P}' = \alpha\beta - \alpha_+\beta_+,
\ea \ee
their linearization about $(Q,P)$ and $(Q',P')$ as
\be \label{eq:LT1} \ba{ll}
\dot{q} = p & \dot{q}'=p'\\
\dot{p} = (I-S^{-1})\alpha\beta_+(S-I)q & \dot{p}' = (I-S^{-1})\alpha_+\beta_+(S-I)q,
\ea \ee
and the B\a cklund transformation as
\be \label{eq:BT1} \ba{l}
F_1:=P+\alpha+\beta - 2\cosh\kappa_m =0\\
F_2:=P'+\alpha+\beta_+ - 2\cosh\kappa_m =0\\
\ea \ee

Linearizing (\ref{eq:BT1}) about $(Q',P',Q,P)$ gives us the linearized B\a cklund transformation
\be p + (\alpha - \beta)q + (\beta S^{-1} - \alpha)q' = p' + (\alpha - \beta_+ S)q + (\beta_+ - \alpha)q' = 0 
\label{eq:LBT1} \ee
which can also be written as $DF_1(q',p',q,p) = DF_2(q',p',q,p) = 0$.

\begin{proposition} \label{pr:LBTcommutes}
Suppose that $(q'(0),p'(0))$ and $(q(0),p(0))$ are related via
the linearized B\"{a}cklund transform $\eqref{eq:LBT1}$.  Then $(q'(t),p'(t))$ and $(q(t),p(t))$ are related via $\eqref{eq:LBT1}$ for all $t \in \R$.
\end{proposition}

\begin{proof}
We claim that there are linear operators $A_i$, $B_i$ such that
\be \dot{DF}_1 = A_1 DF_1 + B_1 DF_2; \qquad \dot{DF}_2 = A_2 DF_1 + B_2 DF_2 \label{eq:no1} \ee
holds.  Assume the claim.  Then since $(DF_1,DF_2)$ satisfies a linear system of equations and $(DF_1(0),DF_2(0)) = (0,0)$, it follows that $(DF_1(t),DF_2(t)) = (0,0)$ for all time.  

It remains to prove the claim.  As the expressions $\dot{\alpha}$ and $\dot{\beta}$ appear in $\dot{DF}_1$ and $\dot{DF}_2$, we differentiate $\eqref{eq:ab}$ with respect to time and make use of $\eqref{eq:BT1}$ to obtain
\[ \dot{\alpha} = \alpha(\beta_+ - \beta); \qquad \dot{\beta} = \beta(\alpha - \alpha_-). \]
A calculation now shows that the choice $A_1 = \alpha - \beta$, $B_1 = \beta S^{-1} - \alpha$, $A_2 = \alpha - \beta_+S$, and $B_2 = \beta_+ - \alpha$ yields equality in $\eqref{eq:no1}$.

\end{proof}

\subsection{Symplectic Orthogonality}

The linearization of the Toda equations about an $m$-soliton solution contains a zero-eigenspace of
dimension $2 m$.  We expect decay of solutions of the linearized equations only if they lie in the
complement of this subspace.  This condition is incorporated in the definition of our function
spaces $X_m(t)$ through the inner products with derivatives of the $m$-solition with
respect to speed and phase.  In order to insure that the sums in these inner products converge
we need some information about the spatial asymptotic behavior of these derivatives.  This
information is contained in the following lemma.

\begin{lemma}[Spatial asymptotics for neutral modes] \label{lem:Cfred}
The operator $x \mapsto x_+ - \alpha \beta_+^{-1}$ is Fredholm with index one and a one-dimensional kernel when regarded as acting on any of the spaces $\ell^2_a$, $\ell^2_{-a}$, or $\ell^2_a \cap \ell^2_{-a}$.  Furthermore the norm of its inverse (acting on the orthogonal complement of its kernel) is bounded uniformly in $t$.  Additionally its kernel is spanned by $\partial_{\gamma_m}Q$.

Moreover $\partial_{\gamma_l} Q$ and $\partial_{\gamma_l} P$ lie in $\ell^2_{-a} \cap \ell^2_a$ and $\partial_{\kappa_l} Q$ and $\partial_{\kappa_l} P$ 
live in $\ell^2_{-a}$ for $a \in (0,2\kappa_m)$ and $1 \le l \le m$.
\end{lemma}
\begin{proof}
We first examine the Fredholm properties of $x \mapsto x_+ - \alpha \beta_+^{-1}$.  This is a non-autonomous linear first order recurrence, which is moreover asymptotically hyperbolic.  At 
minus infinity the origin for $x_+ = e^{2\kappa_m}x$ is unstable (and since it is one-dimensional this means it has a one dimensional unstable manifold).  At infinity, the origin for $x_+ = e^{-2\kappa_m}x$ is stable (and hence has a zero dimensional unstable manifold).  It follows from Palmer's Theorem and its discrete analogs (e.g.  \cite{palmer:1984}) that the difference operator $x \mapsto x_+ - \alpha\beta_+^{-1} x$ is Fredholm with index one and a one-dimensional kernel when regarded as an operator on any of the spaces $\ell^2_{-a}$, $\ell^2_{a}$, or $\ell^2_a \cap \ell^2_{-a}$, so long as the weight doesn't change the stability type of origin at $\pm \infty$, i.e. so long as $a \in [0,2\kappa_m)$.  This establishes the first claim of the lemma.  

More concretely, define the semigroup $T(n) = e^{-\sum_{k=0}^{n-1}(2Q'_k - Q_k - Q_{k+1} - 2\kappa_m)}$ so that $x_{n+1} - \frac{\alpha_n}{\beta_{n+1}}x_n = y_n$ has solution $x_n = T(n)x_0 + \sum_{k=0}^n T(n-k)y_k$.  For simplicity let $x^*_n = \sum_{k-0}^n T(n-k)y_k$ and let $v$ be a unit vector which spans the kernel.  Then we can write $x = x^* + \nu v$ where $\nu = \frac{ \<x^*,v\>}{\|v\|^2}$.   Thus $\|x\| \le (1 + \frac{1}{\|v\|})\|x^*\| = 2\|x^*\|$.  It follows from the Hausdorff-Young inequality for convolutions that $\|x^*\|_a \le \|e^{a\cdot }T\|_{\ell^1} \|y\|_a$ and hence the norm of the inverse restricted to the orthogonal complement of the kernel is bounded above by $2\|e^{a\cdot}T\|_{\ell^1}$, so long as it is finite.  Note that $T$ has the asymptotic form
\be T(n) = \left\{ \ba{ll} \mathcal{O}(e^{-2\kappa_m n}) & n \to \infty \\ \mathcal{O}(e^{2\kappa_m n}) & n \to -\infty \ea \right. \label{eq:Tasy} \ee
so in particular $\| e^{a \cdot} T\|_{\ell^1}$ is finite.  Moreover, in light of Lemma $\ref{lem:resolution}$, as $t \to \pm \infty$, the semigroup $T$ approaches the semigroup $T^\pm$ corresponding to the single-soliton problem.  This semigroup, studied in \cite{hoffman:2008} satisfies $\| e^{a\cdot}T^\pm \|_{\ell^1} < \infty$.  Thus we have the bound $\|e^{a\cdot} T\| < K$ uniformly in time.  This establishes the second claim of the lemma.

We now establish that the derivatives of $Q$ and $P$ lie in $\ell^2_{-a}$.  The proof is based on differentiating $\eqref{eq:BT}$ and solving the resulting first order recurrence.  The details are below.
Let $\partial$ denote $\partial_{\gamma_l}$ or $\partial_{\kappa_l}$ for some $1 \le l \le m$.  Differentiate the second line of $\eqref{eq:BT}$ and solve for $\partial Q_+$ to obtain
\[  
\partial Q_+  =  \alpha\beta_+^{-1}\partial Q + \beta_+^{-1}\partial P' + (1-\alpha\beta_+^{-1}) \partial Q'  + (\alpha \beta_+^{-1} - 1 - 2\sinh\kappa_m \beta_+^{-1}) \partial \kappa_m 
\]
This is an inhomogeneous first order recurrence whose homogeneous part: $\partial Q_+ = \alpha \beta_+^{-1} \partial Q$ we have studied above.  Note that in the case $\partial = \partial_{\gamma_m}$ the inhomogeneous term is zero, thus $\partial_{\gamma_m}Q$ lies in the kernel as desired.  More generally, the inhomogeneous term has the asymptotic form 
$e^{-\kappa_m} \partial P' + (1-e^{-2\kappa_m})\partial Q' - 2(1-e^{-2\kappa_m})\partial \kappa_m$ as $n \to \infty$ and $e^{\kappa_m}\partial P' + (1-e^{2\kappa_m})\partial Q' + e^{-(Q_n'-Q_{n+1})}(e^{2\kappa_m}-1)(e^{-(Q_n'-Q_n)}-1)\partial \kappa_m$ as $n \to -\infty$.  Make the induction hypothesis that $\partial_{\gamma_l} Q'$ and $\partial_{\gamma_l} P'$ lie in $\ell^2_{-a} \cap \ell^2_a$ while $\partial_{\kappa_l} Q'$ and $\partial_{\kappa_l} P'$ lie in $\ell^2_{-a}$.  Noting that the derivative of $\kappa_m$ with respect to any of the parameters $\gamma_l$ or $\kappa_l$ is either one or zero and $Q_n' - Q_n = \mathcal{O}(e^{2\kappa_m n})$ as $n \to -\infty$, it follows that the inhomogeneous term lies in $\ell^2_{-a}$, and hence that $\partial Q$ lies in $\ell^2_{-a}$ as well.  To see that $\partial P \in \ell^2_{-a}$ simply differentiate the first equation in $\eqref{eq:BT}$ to obtain
\[ \partial P  =  ( \beta - \alpha)\partial Q + \alpha \partial Q' - \beta \partial Q'_- + \left(\beta - \alpha + 2\sinh\kappa_m\right)\partial \kappa_m.
 \]
In light of the fact that the coefficient on $\partial \kappa_m$, which is $e^{\kappa_m}(1-e^{-(Q'-Q)}) - e^{-\kappa_m}(1-e^{-(Q-Q_-')})$, goes to zero rapidly as $n \to -\infty$ and is bounded as $n \to \infty$, it follows that $\partial P$ is a linear combination of vectors in $\ell^2_{-a}$, hence in $\ell^2_{-a}$ itself.  

It remains only to verify the base case for the induction.  In the case $m = 1$, we have the vectors $Q' = P' = 0$, hence $\partial Q' = \partial P' = 0 \in \ell^2_{-a}$.  This completes the proof.
\end{proof}

In light of the above together with the Cauchy-Schwartz inequality for the $\ell^2_a$ - $\ell^2_{-a}$ dual pairing it follows that the linear map $(q,p) \mapsto \< q, \partial P \> - \<p, \partial Q\>$ is continuous when regarded as a map from $\ell^2_a \times \ell^2_a \to \R$.  Due to the form of the B\a cklund transformation, it is easier to work with the variables $q$ and $p$ than it is to work with $u$, so before proceeding we will relate the orthogonality conditions in (\ref{eq:subspace}) to $q$ and $p$.

\begin{lemma} \label{lem:orthoequiv}
Let $U=(R,P)=((S-I)Q,P)$ be a $k$-soliton solution of the Toda lattice, and let $u=(r,p)=((S-I)q,p)$.  Then
\[ \< u,J^{-1}\partial U \> = \< p,\partial Q \> - \< q,\partial P \>, \]
where $\partial$ {\color{black} denotes a partial derivative of U with respect to either $\kappa_j$ or $\gamma_j$,
for some $j=1, \dots , k$.} 
\end{lemma}

\begin{proof}
We compute
\[ \ba{rl}
\< u,J^{-1}\partial U\> &= \< (S-I)q,\sum_{k=-\infty}^0 S^k \partial P\> + \< p,\sum_{k=-\infty}^{-1} S^k\partial (S-I)Q\>\\
	&= \< q, (S^{-1}-I)\sum_{k=-\infty}^0 S^k \partial P\> + \< p,\sum_{k=-\infty}^{-1} S^k (S-I)\partial Q\>\\
	&= -\< q,\partial P\> + \< p,\partial Q\>.
\end{array}
\]
Setting both sides to zero proves the lemma.
\end{proof}

We are now able to redefine the spaces $X_k$ in terms of our preferred (for this section) $q$ and $p$ variables:
\[ X_k(t) = \{(q,p) \in \ell^2_a \times \ell^2_a \; | \; \<q,\partial P(t)\> - \<p \partial Q(t)\> = 0 \; \mbox{ for } \partial = \partial_{\gamma_l} \mbox{ or } \partial_{\kappa_l} \mbox{ with } 1 \le l \le k \} \]

\subsection{Linearized B\a cklund transformation preserves orthogonality conditions}
In order to establish the commutativity of the diagram in figure 1, we must show that the map $B$ takes $X_{m-1}$ to $X_m$.  
In particular, we must show that if $(q,p)$ and $(q',p')$ are related by $\eqref{eq:LBT1}$ with $(q',p') \in X_{m-1}$, then $(q,p) \in X_m$. 

It is convenient to define the operators
\be \label{eq:operators}
\ba{ll}
C := \alpha - \beta S^{-1} & \hat{C} := \alpha-S \beta = -\beta_+(S - \alpha\beta_+^{-1})\\
L := \alpha-\beta & M := \alpha -\beta_+,
\ea \ee
where $\alpha$ and $\beta$ are as in the previous section. As we shall see, these operators arise naturally both when differentiating and when linearizing the B\a cklund transformation (\ref{eq:BT}). Since we are interested in the orthogonality conditions from (\ref{eq:subspace}), we differentiate (\ref{eq:BT}) with respect to $\gamma_i$ or $\kappa_i$, $i=1,\ldots,m-1$, to get
\be \label{eq:partial}
\ba{l}
C\partial Q' =  L\partial Q+\partial P\\
\partial P' =  \hat{C}\partial Q + M\partial Q',
\ea \ee
where $\partial=\partial_{\gamma_i}$ or $\partial_{\kappa_i}$, $1\leq i < m$.  These equations help us prove the following:

\begin{proposition} \label{prop:orthoinvariance}
Assume $U'$ and $U$ are related by (\ref{eq:BT}) with $u'$ and $u$ related by $\eqref{eq:LBT1}$.  Then $\< u',J^{-1}\partial U' \> = \< u,J^{-1}\partial U \> ,$ where $\partial=\partial_{\kappa_i}$ or $\partial_{\gamma_i}$ with $1\leq i < m$.
\end{proposition}

\begin{proof}
By lemma \ref{lem:orthoequiv} we need to show that $\< p',\partial Q' \> - \< q',\partial P' \> = \< p,\partial Q \> - \< q,\partial P \>$. Using equations (\ref{eq:LBT2})-(\ref{eq:partial}) we compute
\[ \ba{rl}
\<p,\partial Q \> &= \<Cq'-Lq,\partial Q \>\\
	&= \<q' ,\hat{C}\partial Q\> - \<q,L\partial Q  \>\\
	&= \<q',\partial P' - M\partial Q'\> - \<q,C\partial Q' - \partial P\> \\
	&= \<q',\partial P' - M\partial Q'\> - \<\hat{C}q, \partial Q' \> + \<q,\partial P \> \\
	&= \<q',\partial P' \> - \< Mq' + \hat{C}q, \partial Q' \> + \<q,\partial P\> \\
&= \<q',\partial P' \> - \< p', \partial Q' \> + \<q,\partial P\>, \\
	  \ea \]
establishing the desired identity.
\end{proof}

\subsection{Linearized Toda flow preserves orthogonality conditions}
In this section we consider a solution $(q(t),p(t))$ to $\eqref{eq:LT1}$ which for some $t_0 \in \R$ lies in $X_k(t_0)$.  We show that $(q(t),p(t))$ necessarily lies in $X_k(t)$ for all $t \in \R$.  This implies that the maps corresponding to horizontal arrows in the commuting diagram in figure 1 take $X_k(s)$ to $X_k(t)$ (for $k = m-1,m$).

\begin{proposition} \label{prop:LTorth}
Let $(Q,P)$ be a solution to $\eqref{eq:Toda}$ and let $(q,p) \in \ell^2_a \times \ell^2_a$ be a solution to $\eqref{eq:LT1}$.  Then the quantity
\[ \langle p, \partial Q \rangle - \langle q, \partial P \rangle \]
is independent of time.
\end{proposition}
\begin{proof}
\[ \ba{lll} 
\frac{d}{dt} \langle p, \partial Q \rangle - \langle q, \partial P \rangle & = &  
\langle \dot{p}, \partial Q \rangle - \langle q, \partial \dot{P} \rangle \\ \\
& = & 
\displaystyle{\sum_{n \in \Z} \left\{ \left[ e^{-(Q_{n+1}-Q_n)}(q_{n+1} - q_n) - e^{-(Q_n-Q_{n-1})}(q_n - q_{n-1}) \right] \partial Q_n \right.}
\\ \\ 
& & \displaystyle{- \left. \left[e^{-(Q_{n+1}-Q_n)}(\partial Q_{n+1} - \partial Q_n) - e^{-(Q_n - Q_{n-1})} (\partial Q_n - \partial Q_{n-1}) \right] q_n \right\}} \\ \\
& = & \displaystyle{\sum_{n \in \Z} e^{-(Q_{n+1} - Q_n)}[q_{n+1} \partial Q_n - \partial Q_{n+1} q_n] - e^{-(Q_n - Q_{n-1})}[q_n\partial Q_{n-1} - q_{n-1} \partial Q_n] }\\ \\
& = & 0.
\ea 
\]
In the last line we have used the fact that the sum in the third line of the equation is a telescoping series.  The boundary terms at $n = \pm \infty$ vanish because $q \in \ell^2_a$ and $\partial Q \in \ell^2_{-a}$.
\end{proof}

\subsection{Linearized B\a cklund transformation is an isomorphism}
The linearized B\a cklund tranformation gives us a time-invariant relationship between the linearizations about $(m-1)$- and $m$-soliton Toda solutions, and it preserves the orthogonality conditions from $X_{m-1}$ to $X_m$; the goal of this section is to show that it also implicitly defines an isomorphism $B:(q',p')\mapsto(q,p)$ between subspaces of the form (\ref{eq:subspace}) as in Figure \ref{fig:commuting}.

We can write the linearized B\a cklund transformation in terms of the operators defined in the previous section as
\be \label{eq:LBT2}
\ba{l}
Cq' = Lq + p\\
p' = \hat{C}q + Mq'.
\ea \ee
The goal is to define an isomorphism $B : (q',p') \mapsto (q,p)$ from $X_{m-1}$ to $X_m$.  The map is defined as follows: given $(q',p')$, solve $\hat{C}q = p' - Mq'$ for $q$ and then let $p = Cq' - Lq$.  Similarly, given $(q,p)$, solve $Cq' = Lq + p$ for $q'$ and then let $p' = \hat{C}q + Mq'$.  To implement this, we must first understand the solvability conditions, i.e. Fredholm properties, for the maps $C$ and $\hat{C}$.

Note that it follows from Lemma $\ref{lem:Cfred}$ together with the fact that the multiplication operator $\beta_+$ is bounded with bounded inverse that $\hat{C}$ is Fredholm with index one and a one dimensional kernel when regarded as acting on either $\ell^2_a$ or $\ell^2_{-a}$.  Together with the Fredholm alterative and the fact that $\ell^2_a$ is dual to $\ell^2_{-a}$ this implies that $C$ is Fredholm with index $-1$, injective, and has range given by the orthogonal complement to $\partial_{\gamma_m} Q$ when regarded as acting on $\ell^2_a$.  Since $M$ and $L$ are multiplication operators and the sequences $\alpha$ and $\beta$ are uniformly bounded, it follows that $M$ and $L$ are bounded as well.

Since the range of $C$ is orthogonal to $\partial_{\gamma_m}Q$, it follows that in order to implement the map $B$ we must impose the orthogonality condition $\< Lq + p, \partial_{\gamma_m}Q \> = 0$.
The following lemma relates this orthogonality condition back to the familiar symplectic orthogonality condition that is used in the definition of $X_m$.

\begin{lemma} \label{lem:Lqp}
For $q,p\in \ell^2_a$, $(Lq+p)\in \ran(C)$ if and only if $\< u,J^{-1}\partial_{\gamma_m}U\> = 0$.
\end{lemma}

\begin{proof}
Since $C$ is Fredholm with $\coker(C)=span\{\partial_{\gamma_m}Q\}$, $(Lq+p)\in \ran(C)$ if and only if 
\[
\begin{array}{rl}
0 &= \< Lq+p, \partial_{\gamma_m}Q\> \\
	&= \< q,L\partial_{\gamma_m}Q\> + \< p,\partial_{\gamma_m}Q\>\\
	&= -\< q, \partial_{\gamma_m}P\> + \< p, \partial_{\gamma_m} Q\>
\end{array}
\]
This combined with Lemma \ref{lem:orthoequiv} completes the proof.
\end{proof}

We can also differentiate (\ref{eq:BT}) with respect to $\kappa_m$, using the fact that $\partial_{\kappa_m}Q'=\partial_{\kappa_m}P'=0$, to obtain
\be \label{eq:dk}
L\partial_{\kappa_m}Q+\partial_{\kappa_m}P=\hat{C}\partial_{\kappa_m}Q = 2\sinh\kappa_m.
\ee
These equations help us show that the map $B^{-1}:(q,p) \mapsto(q',p')$ is injective, provided our second orthogonality condition is satisfied.

\begin{lemma} \label{lem:injective}
Suppose (\ref{eq:LBT2}) is satisfied with $q'=p'=0$, and that $\< u,J^{-1}\partial_{\kappa_m}U\>=0$.  Then $q=p=0$.
\end{lemma}

\begin{proof}
It is a consequence of Proposition $\ref{pr:LBTcommutes}$ that the solution $(q(t),p(t))$ to $\eqref{eq:LT1}$ with initial condition $(q,p)$ is related via $\eqref{eq:LBT2}$ to $(q',p') = (0,0)$.  From $\eqref{eq:LBT2}$ we have $\hat{C}q = p' - Mq' = 0$ so 
\be q_n(t) = \mu(t) \partial_{\gamma_m} Q_n(t) \label{eq:lem:in:1} \ee for some $\mu : \R \to \R$.  Similarly 
\be p = Cq' - Lq = - \mu L\partial_{\gamma_m}Q \label{eq:lem:in:2}. \ee  Differentiate $\eqref{eq:lem:in:1}$ with respect to time to obtain
$p = \dot{\mu} \partial_{\gamma_m} Q + \mu \partial_{\gamma_m}P$.  Use $\eqref{eq:lem:in:2}$ to obtain
\[ \dot{\mu} \partial_{\gamma_m} Q = -\mu [ L \partial_{\gamma_m} Q + \partial_{\gamma_m}P ] = 0. \]
The last equality follows from $\eqref{eq:partial}$.  Since $\partial_{\gamma_m} Q(t)$ is non-vanishing as an element of $\ell^\infty$ it follows that $\dot{\mu} \equiv 0$.  We have now established that $(q,p) = \mu(\partial_{\gamma_m}Q,-K\partial_{\gamma_m}Q)$ with $\mu$ independent of $t$.
Using the second orthogonality condition we compute 
\[ \ba{lll} 
0 & = & \displaystyle{\lim_{t \to \infty} \langle p, \partial_{\kappa_m} Q \rangle - \langle q, \partial_{\kappa_m} P \rangle} \\ \\
& = & \displaystyle{-\mu \lim_{t \to \infty}  \langle L \partial_{\gamma_m} Q, \partial_{\kappa_m} Q \rangle - \mu \langle \partial_{\gamma_m} Q, \partial_{\kappa_m} P \rangle} \\ \\
& = & \displaystyle{ -\mu \lim_{t \to \infty}   \langle \partial_{\gamma_m} Q,  (L\partial_{\kappa_m} Q + \partial_{\kappa_m} P) \rangle} \\ \\
& = & \displaystyle{ -\mu (2\sinh \kappa_m) \lim_{t \to \infty}   \sum_{n \in \Z} \partial_{\gamma_m} Q_n } \\ \\
& = & \displaystyle{ -\mu (2\sinh \kappa_m) \lim_{t \to \infty}   \sum_{n \in \Z} \sum_{l = 1}^m \partial_{\gamma_m} Q_n^1(t;\gamma_l,\kappa_l) }\\ \\
& = & \displaystyle{ -\mu (2\sinh \kappa_m) \lim_{t \to \infty}   \sum_{n \in \Z} \partial_{\gamma_m} Q_n^1(t;\gamma_m,\kappa_m) }
\ea 
\]
In the fourth line we have used $\eqref{eq:dk}$.  In the fifth line we have used Lemma $\ref{lem:resolution}$.  Note that $Q_n^1(t;\gamma_m,\kappa_m)$ is a monotone function of $n-ct$ and $\partial_{\gamma_m}Q_n^1$ is the derivative of this function, hence is of one sign.  Thus 
\[ \sum_{n \in \Z} \partial_{\gamma_m} Q_n^1(t;\gamma_m,\kappa_m) \ne 0. \]
This implies that $\mu = 0$ as desired
\end{proof}

\begin{theorem}\label{thm:isomorphism}
Fix $t\in\R$ and define the space
\[
X(t):=\{(q,p)\in\ell^2_a\times\ell^2_a : \< u,J^{-1}\partial_{\gamma_m}U\> = \< u,J^{-1}\partial_{\kappa_m}U\> = 0\},
\]
Let
\[
\begin{array}{rl}
B(t): & \ell^2_a\times\ell^2_a\rightarrow X(t)\\
	&(q'(t),p'(t))\mapsto(q(t),p(t))\\
\tilde{B}(t): & X(t) \rightarrow \ell^2_a\times\ell^2_a\\
	&(q(t),p(t))\mapsto(q'(t),p'(t))\\
\end{array}
\]
be defined implicitly by the equations (\ref{eq:LBT2}).  Then $B(t)$ is an isomorphism with inverse $B^{-1}(t)=\tilde{B}(t)$.
\end{theorem}

\begin{proof}
We will suppress dependence on $t$ throughout this proof.  Suppose $(q,p)\in X$ is given.  By Lemma \ref{lem:Lqp}, $Lq+p$ is in the range of $C$ and so there exists $q'\in\ell^2_a$ satisfying the first equation of (\ref{eq:LBT2}). Since $\ker(C)=\{0\}$, $q'$ is uniquely defined, as is $p':=(Mq'+\hat{C}q)\in\ell^2_a$.  By Lemma \ref{lem:injective}, $\tilde{B}$ is also injective.  

Given $q',p'\in \ell^2_a$, since $\ran(\hat{C})=\ell^2_a$ there exists $q\in\ell^2_a$ satisfying the second equation of (\ref{eq:LBT2}).  Then $p\in\ell^2_a$ is defined by the first equation of (\ref{eq:LBT2}). Moreover $Lq+p\in \ran(C)$, and since $\< \partial_{\kappa_m} U,\partial_{\gamma_m} U \> \neq 0$ for $m\geq 1$ we can shift $u$ by a multiple of $J^{-1}\partial_{\gamma_m}U$ to get $u \bot J^{-1}\partial_{\kappa_m} U$. Thus, $\tilde{B}$ is surjective and therefore an isomorphism with inverse $B$.
\end{proof}

\begin{corollary} \label{cor:Xm_isomorphism} 
$B(t):X_{m-1}(t)\rightarrow X_m(t)$ is an isomorphism.
\end{corollary}

\begin{proof}
This follows from Proposition \ref{prop:orthoinvariance} and the fact that $X_{m-1}(t)$ and $X_m(t)$ are codimension $2(m-1)$  subspaces of $\ell^2_a\times\ell^2_a$ and $X(t)$, respectively.
\end{proof}

In the following we will need to show that $B(t)$ and $B(t)^{-1}$ are bounded uniformly in time.  

\begin{lemma} \label{lem:Bbound} The maps $B(t)$ and $B^{-1}(t)$ are uniformly bounded for $t\in\R$.
\end{lemma}

\begin{proof}
This follows from the uniform boundedness of the operators $L$, $M$, $C$ and $\hat{C}$, as well as the inverses of $C$ and $\hat{C}$, appropriately restricted.  $L$ and $M$ are multiplication operators by sequences which are uniformly bounded, hence their operator norms are uniformly bounded.  $C$ and $\hat{C}$ are the sum of a multiplication operator by a uniformly bounded sequence and a shift composed with a multiplication operator by a uniformly bounded sequence, hence their operator norms are bounded uniformly too.  The uniform boundedness of $\hat{C}^{-1}$ was addressed in Lemma $\ref{lem:Cfred}$.  The uniform boundedness of $C^{-1}$ is similar.
\end{proof}

\subsection{Linear stability of the Toda flow}
To prove Theorem \ref{thm:main} it only remains to show that the decay rates of $\Phi_{m-1}(t,s)$ are inherited by $\Phi_m(t,s)$  in Figure \ref{fig:commuting}.  

\begin{proposition} Let $c>1$, $a\in(0,2\kappa)$ with $\kappa=\min_{1\leq i\leq m} \kappa_i$, and $\beta = ca-2\sinh(a/2)$.  Suppose $u'(t)\in X\subset \ell^2_a \times \ell^2_a$ satisfies
\[
u'_t = JH''(U')u',
\]
where as above $U'=((S-I)Q',P')$ is an $m-1$ soliton solution with solitons moving at speeds $\kappa_1,\ldots,\kappa_{m-1}$.  Suppose further that for some constant $\tilde{K}>0$ we have
\[ ||e^{a(n-ct-T)}u'(t)|| \leq \tilde{K}e^{-\beta(t-s)}||e^{a(n-cs-T)}{\color{black} u'(s)}||. \]
Then for any $u\in Y:=\{u\in X| \<u,J^{-1}\partial_{\gamma_m}U\> = \<u,J^{-1}\partial_{\kappa_m}U\> = 0\}$, where $U$ is an $m$-soliton solution with solitons moving at speeds $\kappa_1,\ldots,\kappa_m$ and is related to $U'$ by (\ref{eq:BT}), there exists a constant $K>0$ such that
\[ ||e^{a(n-ct-T)}u(t)|| \leq Ke^{-\beta(t-s)}||e^{a(n-cs-T)}u(s)||\ , \]
{\color{black} where $u(t)$ is a solution of the Toda equations, linearized about the 
m-soliton solution $U(t)$.}
\end{proposition}

\begin{proof}  {\color{black} This follows from the fact that $u(t) = B(t) u'(t)$ (see Figure 1),
the bound on $u'$ in the hypothesis of the Proposition,  and Lemma \ref{lem:Bbound}.} \end{proof} 

The proof of Theorem $\ref{thm:main}$ now follows from induction on $m$.  The initial step is provided by \cite{mizumachi:2008}, Lemma 10 and the above lemma proves the inductive step.

\section{Nonlinear Stability}
The purpose of this section is to leverage the linear stability result, Theorem $\ref{thm:main}$ proven in the last section, in order to prove the nonlinear stability result Theorem $\ref{thm:stab}$.  The strategy of proof is a natural generalization of \cite{pego:1994,friesecke:2002} to the setting of multi-soliton solutions.  Roughly speaking, in section 3.1 we show that so long as we are willing to let the parameters $\kappa_i$ and $\gamma_i$ modulate, then we can impose an orthogonality condition on any (sufficiently small) perturbation of an $m$-soliton solution; in section 3.2 we derive coupled ODEs for the modulating parameters $\kappa$ and $\gamma$ and the perturbation $v$; in section 3.3 we obtain estimates which allow us to determine the long-time behavior of solutions to these ODEs with small initial data via the method of bootstrapping.

\subsection{$(v,\xi)$ Coordinates}

Let $\kappa = (\kappa_1,\cdots \kappa_m)$ and $\gamma = (\gamma_1, \cdots \gamma_m)$ be functions of time which we will specify later.  Let $\xi = (\gamma_1,\kappa_1,\gamma_2,\kappa_2,\cdots,\gamma_m,\kappa_m)$.  Let $v$ be defined by the equation 
\begin{equation}\label{eq:vdef}
u = U(t;\kappa(t),\gamma(t)) + v(t) \equiv U(t;\xi(t)) + v(t)\ ,
\end{equation}
where as in the preceding section, $U$ is an $m$-soliton solution of the Toda system.

We compute
\be [\partial_t - JH''(U(t;\kappa(t_0),\gamma(t_0)))]v = JR - \sum_{j=1}^{2m} \partial_{\xi_i} U \dot{\xi}_i. \label{eq:main} \ee
Here $R = R_1 + R_2$ with $R_1 = H'(U + v) - H'(U) - H''(U)v$, 
and $R_2 = [H''(U(t;\xi(t_0)))-H''(U(t;\xi(t)))]v$.  We work in the space $X_m$ as defined in $\eqref{eq:subspace}$.

\begin{lemma}[Symplectic Form Restricted to the Tangent Space] \label{lem:tan}
Let 
\[ \mathcal{A}_{ij}(t) := \langle J^{-1} \partial_{\xi_i}U(t,\xi_*),{\color{black} \partial_{\xi_j}}U(t,\xi_*)\rangle. \]
Then $\mathcal{A}$ is independent of $t$.   Moreover, $\mathcal{A}$ is block diagonal with $2 \times 2$ blocks given by 
\[ \left( \ba{ll} 0 & \alpha_{0,k} \\ \\ -\alpha_{0,k} & \alpha_{1,k} \ea \right) \] 
where $\alpha_{0,k} = \frac{1}{c_k} \frac{d}{dc}H(u_{c_k})\frac{dc}{d\kappa}$ and $\alpha_{1,k} = \left(\frac{d}{dc} \int_\R r_{c_k}\right) \frac{d}{dc} \left( c \int_\R r_{c_k} \right)\left(\frac{dc}{d\kappa}\right)^2$.  
\end{lemma}
\begin{proof}
Observe
\[ \ba{lll}
\frac{d}{dt} \mathcal{A}_{ij} & = & \langle J^{-1} \partial_{\xi_i} \partial_t U, \partial_{\xi_i} U \rangle + \langle J^{-1} \partial_{\xi_i}U, \partial_{\xi_i} \partial_t U \rangle \\ \\
& = & \langle J^{-1} J H''(U)\partial_{\xi_i}U, \partial_{\xi_j} U \rangle + \langle J^{-1} U_{\xi_i}, JH''(U)\partial_{\xi_i}U \rangle
 \\ \\
& = & \langle H''(U)\partial_{\xi_i} U - H''(U)\partial_{\xi_i} U, \partial_{\xi_j} U \rangle = 0 
\ea 
\]
Here we have used the fact that $H''(U)$ is self-adjoint and also that the form $\langle \cdot, J^{-1} \partial_{\xi_i} U \rangle$ is skew-symmetric when acting on zero mean sequences.  Evaluate $\mathcal{A}$ as $t \to \infty$ and use Lemma 2.1 in \cite{friesecke:2002} to see that $\mathcal{A}$ is block diagonal with $2 \times 2$ blocks given as in the statement of the lemma.
\end{proof}

{\color{black}
\begin{corollary} ${\cal A}(t)$ is invertible with uniformly bounded inverse.
\end{corollary}
\begin{proof} This follows from the fact that $\alpha_{0,k} \ne 0$ for all $k$, a fact that is proven
in \cite{friesecke:1999}.
\end{proof}
}

\begin{proposition}[Tubular Coordinates]
There is a $\delta_* > 0$ such that $(v, \xi)$ coordinates correspond to a unique solution $u$ of the Toda lattice equations in a neighborhood $\|v\| + |\xi-\xi_*| < \delta_*$.
\end{proposition}
\begin{proof}  {\color{black} We proceed as in \cite{friesecke:1999}.}
Define $F : \R^{2m} \times \R \times \ell^2_a \to \R^{2m}$ by
$F_j(\xi,t,u) = \langle J^{-1}\partial_{\xi_j} U(t,\xi),u-U(t,\xi) \rangle$.
Let $\xi_* \in \R^{2m}$ and $t_* \in \R$ be given.
Observe that $F(\xi_*,t,U(t,\xi_*)) = 0$.  
Compute $ \partial_{\xi_i} F_j(\xi_*,t,U(t,\xi_*)) = -\mathcal{A}_{ij}$ with $\mathcal{A}$ given as in Lemma $\ref{lem:tan}$.
In particular $D_\xi F(\xi_*,t,U(t,\xi_*))$ is invertible.  It now follows from the implicit function theorem that there is a smooth function $(u,t) \mapsto \xi(u,t)$ mapping a neighborhood of $U(0,\xi_*) \times \{t_*\}$ in $\ell^2_a \times \R$ to a neighborhood of $\xi_*$ in $\R^{2m}$ such that $u-U(t,\xi(u,t)) \in (\ell^2_a)^\perp(t)$.  Since the derivative $\partial_{\xi_i}F_j(\xi_*,t,U(t,\xi_*))$ is independent of time, so is the neighborhood of $\xi_*$ on which the function $\xi$ is defined.  Moreover, the map $(u,t) \mapsto (\xi,v)$ given by $\xi = \xi(u,t)$ and $v = u-U(t,\xi(u,t))$ is, for each fixed $t$, an isomorphism from $\ell^2_a$ to $(\ell^2_a)^\perp(t) \times \R^{2m}$.

\end{proof}

\subsection{Modulation Equations}
Take the inner product with $J^{-1} \partial_{\xi_j} U$ in $\eqref{eq:main}$ to obtain
\be \langle (\partial_t - JH''(U))v, J^{-1}\partial_{\xi_j} U \rangle -  \langle JR,J^{-1}\partial_{\xi_j}U \rangle =  \sum_{i=1}^{2m} \langle \partial_{\xi_i} U, J^{-1} \partial_{\xi_j} U \rangle \dot{\xi}_i. \label{eq:mod1} \ee
Differentiate the identity $\langle v(t), J^{-1} \partial_{\xi_j} U(t;\xi(t_0)) \rangle \equiv 0$ with respect to time to obtain 
\[ \langle [\partial_t -JH'(U)]v,J^{-1}\partial_{\xi_j} U \rangle \equiv 0 \]
thus the modulation equation $\eqref{eq:mod1}$ reduces to
\[ \mathcal{A}\dot{\xi} = b \]
where 
$b_{j} = -\langle JR, J^{-1} \partial_{\xi_j}U \rangle = \langle R, \partial_{\xi_j}U \rangle$
and $\mathcal{A}$ is given as in Lemma $\ref{lem:tan}$.  

Thus
\begin{equation} |\dot{\xi}_j| \le  \|\mathcal{A}^{-1}\| \left|\langle R, \partial_{\xi_j} U + \partial_{\xi_{j\pm1}} U \rangle\right| \le K\|R\|_a. \label{eq:xidot} \end{equation}
The last of these inequalities used the Cauchy-Schwartz inequality, plus the estimates on derivatives
of the $m$-soliton established in lemma \ref{lem:Cfred}

\subsection{Stability estimates}
\begin{proposition}
Let $M > 0$ be given.  Then there exists a constant $K$ such that \begin{equation}
\| R \|_a \le  K\left(\|v\|_{\ell^\infty} + |\xi(t) - \xi(t_0)| \right)\|v\|_a \label{eq:Rest}
\end{equation}
holds so long as $\|v\|_a + \|v\|_{\ell^\infty} \le M$.
\end{proposition}

\begin{proof}
We estimate $\| R_1 \|_a  \le  K\|v\|_a\|v\|_{\ell^\infty}$ using Taylor's theorem with remainder.  We estimate $\| R_2 \|_a \le  K\|v\|_a(|\xi(t) - \xi(t_0)|)$ using the fact that $H''$ is locally Lipschitz.
\end{proof}

In order to proceed we must check that $\|v\|_{\ell^\infty}$ (which is controlled by $\|v\|$) remains well-behaved for long times.  As a preliminary step we prove
\begin{lemma}
\label{eq:dUdt}
\[ \langle H'(U(t;\kappa(t_0),\gamma(t_0))),v(t) \rangle \equiv 0. \]
\end{lemma}
\begin{proof}
The proof follows from the following identity satisfied by the $m$-soliton profile:
\begin{lemma}\label{lem:profile_identity}
If $U$ is an $m$-soliton solution of the Toda equations, then
\begin{equation}
\partial_t U = -\sum_i 2 \sinh \kappa_i \partial_{\gamma_i} U \ .
\end{equation}
\end{lemma}
This lemma is proved in the appendix - assuming that it holds, 
 and recalling the orthogonality condition $\langle J^{-1} \partial_{\gamma_i} U, v \rangle \equiv 0$ it follows that $\langle J^{-1} \partial_t U, v \rangle \equiv 0$, thus $\langle H'(U), v \rangle \equiv 0$, as desired.
 
 \end{proof}

We now derive our first estimate on $v(t)$, the perturbation of the $m$-soliton.

\begin{proposition}
There is a constant $K$ such that if $v(t)$ is the function defined
in \eqref{eq:vdef}, then 

\be  \|v(t)\|^2  \le  K\left(\|v(t_0)\|^2 + |\xi(t) - \xi(t_0)| \right) 
\label{eq:l2}
\ee
holds for $t \ge t_0$ so long as $\|v(t)\| \le \delta_*$.
\end{proposition}

\begin{proof}
Let $u(t) := U(t;\kappa(t),\gamma(t)) + v(t)$.  From the form of the Hamiltonian function
we know
\[ K_-\|v\|^2 \le H(u) - H(U) - \langle H'(U), v \rangle \le K_+\|v\|^2. \] 
Thus
\[ \ba{lllr}
K_-\|v(t)\|^2 
& \le & 
H(u(t)) - H(U(t;\xi(t))) - \langle H'(U(t;\xi(t))), v(t) \rangle \\ \\
& = &  H(u(t)) - H(u(t_0)) & (i) \\ \\ 
& & - H(U(t;\xi(t))) + H(U(t;\xi(t_0))) & (ii) \\ \\ 
& & - H(U(t; \xi(t_0))) + H(U(t_0;\xi(t_0))) & (iii) \\ \\
& & + \langle H'(U(t;\xi(t_0))),v(t)\rangle - \langle H'(U(t;\xi(t))),v(t) \rangle & (iv) \\ \\
& & + \langle H'(U(t_0;\xi(t_0))),v(t_0) \rangle - \langle H'(U(t;\xi(t_0))),v(t)\rangle & (v) \\ \\
& & + H(u(t_0)) - H(U(t_0;\xi(t_0))) - \langle H'(U(t_0;\xi(t_0))),v(t_0) \rangle & (vi) 
\ea
\]
We estimate
$(i) = 0$ by conservation of the Hamiltonian,
$(ii) \le C|\xi(t)-\xi(t_0)|$ because $H$ and $U$ are locally Lipschitz in their arguments.
$(iii) = 0$ by conservation of the Hamiltonian
$(iv) \le K|\xi(t)-\xi(t_0)|\|v(t)\| \le K|\xi(t)-\xi(t_0)|$,
$(v) = 0$ by Lemma \ref{eq:dUdt},
$(vi) \le K\|v(t_0)\|^2$.  
Summing these estimates yields $\eqref{eq:l2}$.
\end{proof}

\begin{lemma}
Let $\delta_0$, $\delta_1$ and $\delta_2$ be positive numbers.  Suppose that the following estimates hold for $t\in[t_0,t_1]$:
\begin{equation}
|\xi(t) - \xi(t_0)| < \delta_0; \qquad \|v(t)\|_a < \delta_1 e^{-\beta' (t-t_0)}; \qquad \|v(t)\|< \delta_2.
\label{eq:stab1}
\end{equation}

Then in fact, the estimates $|\xi(t) - \xi(t_0)| < \frac{K(\delta_2 + \delta_0)\delta_1}{\beta}$, $\|v(t)\| \le K(\|v(t_0)\| + \sqrt{\delta_0})$, and $\|v(t)\|_a \le \delta_1 e^{-\beta'(t-t_0)}\left(\frac{K(\delta_2 + \delta_0)}{\beta - \beta'}  + \frac{K}{\delta_1}e^{-(\beta-\beta')(t-t_0)}\|v(t_0)\|_a\right)$ hold for all $t \in [t_0,t_1]$.
\end{lemma}
\begin{proof}
Substitute $\eqref{eq:stab1}$ into $\eqref{eq:l2}$ to obtain $\|v(t)\|^2 \le K(\|v(t_0)\|^2 + \delta_0)$.  Substitute $\eqref{eq:stab1}$ into $\eqref{eq:Rest}$ and then into $\eqref{eq:xidot}$ to obtain $|\dot{\xi}| \le K\delta_1(\delta_2 + \delta_0)e^{-\beta'(t-t_0)}$.  Integrating yields $|\xi(t)-\xi(t_0)| \le \frac{K\delta_1(\delta_2 + \delta_0)}{\beta'}$.  To estimate $\|v(t)\|_a$ again substitute $\eqref{eq:stab1}$ into $\eqref{eq:Rest}$, then rewrite the evolution equation $\eqref{eq:main}$ using the variation of constants formula and make use of the linear decay estimates in Theorem $\ref{thm:main}$ to obtain
\[ \ba{rl} \|v(t)\|_a & \le Ke^{-\beta(t-t_0)}\|v(t_0)\| + \int_0^t e^{-\beta(t-s)}K\delta_1(\delta_2+\delta_0)e^{-\beta's} ds \\
	& = \delta_1e^{-\beta'(t-t_0)}\left[\frac{K(\delta_0 + \delta_2)}{\beta - \beta'} + \frac{K}{\delta_1} e^{-(\beta-\beta')(t-t_0)}\|v(t_0)\|_a\right] \ea \]
as desired.
\end{proof}

\begin{proof}[Proof of Theorem 2]
Let $\delta_2 > 0$ be given.  Let $\delta_0 < \frac{\delta_2^2}{2K}$ and let $\delta_1 < \frac{\delta_0 \beta'}{K(\delta_2 + \delta_0)}$.  Finally, restrict attention to $v(t_0)$ so small so that $\|v(t_0)\| < \frac{\delta_2}{\sqrt{2K}}$ and $\|v(t_0)\|_a < \frac{\delta_1}{2K}$.  Apply Lemma 4.5 to conclude that we can take $t_1 = \infty$ without loss of generality.  To establish the fact that $\kappa_i$ and $\gamma_i$ converge to limits at $\pm \infty$ note that from what we have already proven $\|v(t)\|_a$, and hence $\|R\|_a$ decay exponentially fast.  It now follows from $\eqref{eq:xidot}$ that $\dot{\xi}_j$ decays exponentially as well and hence that $\xi_j(t)$ has a limit as $t \to \infty$.
\end{proof}

{\bf Acknowledgements:} This work was funded in part by the National Science Foundation under grants DMS-0603589 and DMS-0908093.  The first author was supported in part by grant DMS-0602204
EMSW21-RTG, BIODYNAMICS AT BOSTON  UNIVERSITY.  The second and third authors also wish to thank R. Pego and T. Mizumachi
for useful discussions concerning B\"acklund transformations and stability.

\section{Appendix: Form of the m-soliton solutions of the Toda model}\label{sec:TodaFacts}

In the present appendix we prove various facts about the m-soliton solutions of the Toda model.
While these facts will probably not be surprising to experts, we were unable to find their proofs in
the literature and thus collect them here for future reference.

Proof of Lemma \ref{lem:resolution} (Resolution of $m$-soliton solutions into one-solitions).  
\begin{proof}
While 
the asymptotic resolution  into one-solitons is one of the defining characteristics of an $m$-soliton
solution this lemma shows that the decomposition takes place exponentially rapidly in the space
$\ell^1$ (and hence in all other $\ell^p$ spaces.)  We begin by recalling the explicit expression for
the $m$-soliton (\cite{toda:1989}, Section 3.6)  namely
\begin{equation}\label{eq:Qm_formula}  
Q^m = Q^m(\gamma_1,\cdots \gamma_m,\kappa_1,\cdots \kappa_m) = (I - S^{-1})\log d
\end{equation}
where  $d = \det(I + C)$ and 
\begin{equation}\label{eq:C_formula}
C_{ij}=C_{ij}(n,t)  = \frac{e^{-(\eta_i(n,t) + \gamma_i + \eta_j(n,t) + \gamma_j)}}{1 - e^{-(\kappa_i + \kappa_j)}}\ .
\end{equation}
with $\eta_j(n,t)  = \kappa_j n - \sinh(\kappa_j) t$.
Note that this notation also encompasses the formula one-soliton
solutions which can be 
written as $Q^1(\gamma,\kappa) = (I - S^{-1}) \log (1 + \frac{e^{-2(\eta + \gamma)}}{1 - e^{-2\kappa}})$,
with $\eta(n,t) = \kappa n - \sinh(\kappa) t$.

We begin by noting that there is a simple, explicit formula for the phase shifts,
$\zeta_j^\pm$, undergone by the constituent
solitons in the $m$-solition as they interact with one another.  
Let $\alpha_{ij} = \frac{1}{1 - e^{-(\kappa_i + \kappa_j)}}$, let $\alpha^{k,+}$ denote the $(m-k+1) \times (m-k +1)$ matrix with entries $\alpha_{ij}$ for $k \le i,j \le m$, and let $\alpha^{k,-}$ denote the $k \times k$ matrix with entries $\alpha_{ij}$ for $1 \le i,j \le k$.  Since $\kappa_i > 0$ for each $i$, it follows that each $\alpha_{ij}$ is positive, hence that each matrix $\alpha^{k,\pm}$ is positive-definite, hence that $\det \alpha^{k,\pm}$ is positive.  

Thus we may define $\zeta_j^+ = \frac{1}{2} \log \frac{\det \alpha^{j+1,+}}{\det \alpha^{j,+}}$ for $1 \le j < m$ with $\zeta_m^+ = -\frac{1}{2}\log \det \alpha^{m,+}$ and $\zeta_j^- = \frac{1}{2}\log \frac{\det \alpha^{j-1,-}}{\det \alpha^{j,-}}$ for $1 < j \le m$ with $\zeta_1^- = -\frac{1}{2}\log \det \alpha^{1,-}$.

We now proceed to the proof of $\eqref{eq:res}$.  We write $Q^m = Q^m(\gamma_1,\cdots \gamma_m,\kappa_1,\cdots \kappa_m) = (I - S^{-1})\log d$ where $d = \det(I + C)$ and $C_{ij} = \frac{e^{-(\eta_i + \gamma_i + \eta_j + \gamma_j)}}{1 - e^{-(\kappa_i + \kappa_j)}}$.  Similarly we write $Q^1(\gamma,\kappa) = (I - S^{-1}) \log (1 + \frac{e^{-2(\eta + \gamma)}}{1 - e^{-2\kappa}})$

\[ \ba{lll} 
Q^m - \sum_{i=1}^m Q^1(\gamma_i + \zeta_i^\pm,\kappa_i) & = & \displaystyle{(I - S^{-1})\log\left(\sum_{\sigma \in S^m} \mathrm{sgn}(\sigma) \Pi_{j=1}^n \frac{\delta_{j,\sigma(j)} + C_{j,\sigma(j)}}{1 + e^{-2\zeta^\pm_j}C_{jj}}\right)} \\ \\ 
& = & \displaystyle{(I - S^{-1})\log\left( 1 +  \Delta^\pm \right)}
\ea
 \]
where \[\Delta^\pm = \frac{\left[ \sum_{\sigma \in S^m} \mathrm{sgn}(\sigma) \Pi_{j=1}^m \delta_{j,\sigma(j)} + C_{j,\sigma(j)}\right] - \left[ \Pi_{j=1}^m (1 + e^{-2\zeta_j^\pm} C_{jj}) \right]}{\Pi_{j=1}^m (1 + e^{-2\zeta^\pm_j} C_{j,j})}.\]  In light of the fact that $(I - S^{-1})\log(1 + x) = \log(1 + \frac{x-S^{-1}x}{1 + S^{-1}x})$ it suffices to show that $\| \Delta^\pm \|_{\ell^1} \to 0$ exponentially fast as $t \to \pm \infty$.  Let $\xi_j = e^{-\eta_j-\gamma_j}$ so that $C_{ij} = \xi_i\xi_j \alpha_{ij}$.  Then we can write
\[
\ba{lll} 
\Delta^\pm & = & \displaystyle{\frac{\sum_{J \subset \{1,\cdots m\}} (\Pi_{k \in J} \xi_k^2) \left[\left(\sum_{\{ \sigma \in S^m \; | \; J^C \subset \mathrm{Fix}\sigma\}} \mathrm{sgn}(\sigma) \Pi_{j \in J}\alpha_{j,\sigma(j)}\right) -\Pi_{k \in J} e^{-2\zeta^\pm_k}\right]} {\Pi_{j=1}^m (1 + e^{-2\zeta^\pm_j} \alpha_{j,j} \xi_j^2 )}} \\ \\ 
& = & 
\displaystyle{\frac{\sum_{J \subset \{1,\cdots m\}} (\Pi_{k \in J} \xi_k^2) \left[\det \alpha^J - e^{-2\sum_{k \in J}\zeta^\pm_k}\right]} {\sum_{J \subset \{1,\cdots m\}} \det \alpha^J e^{-2\sum_{k \in J}\zeta^\pm_k} \Pi_{k \in J} \xi_k^2 }}
\ea
\]
where the sum runs over all subsets of the set of
integers $\{ 1, 2, \dots , m\}$, and  $\alpha^J$ is the $|J| \times |J|$ matrix with entries $\alpha_{ij}$ for $i,j \in J$.

Note that for any fixed $t > 0$ and $n$, the polynomials $\prod_{k \in J} \xi_k^2 = \prod_{k \in J} \xi_k(n,t)^2$ can
be ordered in size and the largest will correspond to a subset $J$ in which $\xi_k > 1$ if
and only if $k \in J$.  This follows from the observation that for each $n$, if $\xi_j(n,t) > 1$, then
$\xi_k(n,t) > 1$ for $k = j+1, \dots, m$ and if $\xi_j(n,t) < 1$, then $\xi_k(n,t) < 1$ for $k=1,\dots , j-1$.
If $t< 0$, a similar argument shows that the dominant polynomial corresponds to 
$J = \{1,2,\cdots k\}$ for some $k$ (as well as the empty set). 
For the remainder of the proof, we concentrate on the case $t>0$ -- the case when $t \to -\infty$
is handled in a similar fashion.

The key observation is that in the numerator of the expression for $\Delta^+(n,t)$, the definition
the asymptotic phase shifts, $\zeta^{\pm}$,  insures that the coefficient of the largest polynomial is zero.
This is sufficient to prove the lemma.

In more detail, set $n_j(t) = \lfloor \frac{\sinh(\kappa_j) t)}{\kappa_j} \rfloor$, where $\lfloor x \rfloor$ denotes
the integer part of $x$.  Also, set $\lambda(t) = \min_{j=1,\dots ,m-1} (n_{j+1}(t) - n_j(t))$.  Note that there
exists $\lambda_0 > 0$ such that $\lambda(t) \ge \lambda_0 t$ for $t$ sufficiently large.
Now let $J_{max} = \{ k^*, k^*+1, \dots , m\}$ denote the index set for the dominant terms in $\Delta^{\pm}$
and consider any quotient of the form
\begin{equation}\label{eq:quotient}
\frac{ \prod_{j \in J} \xi_j^2(n,t) }{ \prod_{j \in J_{max} } \xi_j^2(n,t) }
\end{equation}
for which the coefficient in the numerator of $\Delta^{\pm}$ is non-zero.  Note that $J \ne \{ k^*+1, k^*+2, \dots , m\}$
or $ \{ k^*-1, k^*, \dots , m\}$ since such terms also have zero coefficients due the definition of $\zeta^{\pm}$.
Hence, for any nonzero term, there is either an ``extra'' factor of the form $\xi_j^2(n,t)$, with $j < k^*-1$ in
the numerator of \eqref{eq:quotient}, or a ``missing'' factor of $\xi_j^2(n,t)$, with $j > k^*+1$.  If we
define $\Gamma = \max | \gamma_j |$, then from the fact that $\kappa_1 \le \kappa_j \le \kappa_m$, 
and the definition of $\xi_j(n,t)$, we see that the quotient in \eqref{eq:quotient} can be bounded by
\begin{equation}\label{eq:quotient_bound}
\left| \frac{ \prod_{j \in J} \xi_j^2(n,t) }{ \prod_{j \in J_{max} } \xi_j^2(n,t) } \right|
\le e^{2\kappa_m} e^{2\Gamma} e^{-\kappa_1 |n-n_j(t)  |} e^{-\kappa_1 \lambda(t)}
\end{equation}
for some $j = 1, \dots , m$.  Since there are only finitely many term in the numerator and denominator
of $\Delta^{\pm}$, we see that we have the bound
\begin{equation}
| \Delta^{\pm}(n,t) | \le K e^{2\kappa_m} e^{2\Gamma} e^{-\kappa_1 \lambda(t)}
 \sum_{j=1}^m   e^{-\kappa_1 |n-n_j(t)  |}  \ .
\end{equation}
Summing over $n$ then completes the proof of Lemma \ref{lem:resolution}. 
\end{proof}

We now turn to Proof of Lemma \ref{lem:profile_identity} which is an identity satisfied by the profile of the
$m$-soliton solution.
\begin{proof}
Note that it suffices to show that the scalar equality $\partial_t r = \sum_i 2 \gamma_i \sinh \kappa_i \partial_{\gamma_i} r$ for all time.  This is because $p(t)$ can be recovered from $r(t)$ via  $p = (S-I)^{-1}\dot{r}$ and the transformation $r \mapsto p$ commutes with time derivatives, $\gamma_i$ derivatives, and linear combinations thereof.  Thus if the identity holds for $r$, it holds also for $p$ and hence for $U = (r,p)$.

We write $r = \Delta \log d$ where $d = \det(I + C)$ with $C_{ij} = \frac{1}{1 - e^{-(\kappa_i + \kappa_j)}} e^{-(\eta_i +\gamma_i + \eta_j + \gamma_j)}$ and $\eta_i = \kappa_i n - \sinh \kappa_i t$.
Thus $\partial_\diamond r = \Delta \frac{\partial_\diamond d}{d}$ where $\partial_\diamond$ stands for $\partial_t$ or $\partial_{\gamma_i}$ for some $i$.  Here we have used the fact that the discrete Laplacian $\Delta$ commutes with derivatives with respect to $t$ and $\gamma_i$.
Write $d = \sum_{\sigma \in S_n} \mathrm{sgn}(\sigma) \Pi_{j = 1}^n (\delta_{j\sigma(j)} + C_{j\sigma(j)})$ so that
\[ \partial_{\diamond} d = \sum_{\sigma \in S_n} \mathrm{sgn}(\sigma) \sum_{l = 1}^n \partial_{\diamond} C_{l,\sigma(l)} \Pi_{j \ne l} (\delta_{j\sigma(j)} + C_{j\sigma(j)}) \]

Note that $\partial_{\gamma_i} C_{jk} = -\frac{\delta_{ij} + \delta_{ik}}{2} C_{jk}$ and that $\partial_t C_{jk} = (\sinh(\kappa_j) + \sinh(\kappa_k))C_{jk}$.
Thus
\[ \ba{lll} 
-\sum_{i = 1}^n 2 \sinh(\kappa_i) \partial_{\gamma_i} d & = & 
\displaystyle{
\sum_{i=1}^n \sum_{\sigma \in S_n} \mathrm{sgn}(\sigma) \sum_{l = 1}^n (\delta_{i,l} + \delta_{i,\sigma(l)})\sinh(\kappa_i) C_{l,\sigma(l)} \Pi_{j \ne l} (\delta_{j\sigma(j)} + C_{j\sigma(j)})
} \\ \\
& = & \displaystyle{
\sum_{l=1}^n \sum_{\sigma \in S_n} \mathrm{sgn}(\sigma) (\sinh(\kappa_l) + \sinh(\kappa_{\sigma(l)}))C_{l,\sigma(l)} \Pi_{j \ne l} (\delta_{j\sigma(j)} + C_{j\sigma(j)})
} \\ \\
& = & \displaystyle{
\sum_{\sigma \in S_n} \mathrm{sgn}(\sigma) \sum_{l=1}^n \partial_t C_{l,\sigma(l)} \Pi_{j \ne l} (\delta_{j\sigma(j)} + C_{j\sigma(j)})
} \\ \\
& = & \partial_t d
\ea
\]

\end{proof}

\bibliography{BHW_reference}

\end{document}